\documentclass{article}
\usepackage[utf8]{inputenc}
\usepackage{amsmath, amssymb, amsthm,mathrsfs}
\usepackage{hyperref}
\usepackage{autobreak}
\usepackage[all]{xy}    %%図式用パッケージ
\usepackage{graphicx}   %%画像用パッケージ
\usepackage{authblk}    %%タイトル用
\usepackage{url}    %%URL のハイパーリンクを使えるようにする

\allowdisplaybreaks[3] %%数式中でも改ページ可能に. 数字が大きいほど改ページされやすい

\makeatletter

\@addtoreset{equation}{subsection}
\makeatother

\newcommand{\ZZ}{\mathbb Z}
\newcommand{\QQ}{\mathbb Q}

\newcommand{\CC}{\mathbb C}
\renewcommand{\AA}{\mathbb A}

\newcommand{\FF}{\mathbb F}

\newcommand{\bsym}{\boldsymbol}

\newtheorem{lemm}{Lemma}[section]
\newtheorem{thm}[lemm]{Theorem}
\newtheorem{prop}[lemm]{Proposition}
\newtheorem{coro}[lemm]{Corollary}
\newtheorem{rem}[lemm]{Remark}

\newcommand{\floor}[1]{\left\lfloor #1\right\rfloor}
\newcommand{\ceil}[1]{\left\lceil #1\right\rceil}

\DeclareMathOperator{\Spec}{\mathrm{Spec}}
\newcommand{\SL}{\mathrm{SL}}
\newcommand{\GL}{\mathrm{GL}}
\newcommand{\Hom}{\mathrm{Hom}}

\newcommand{\stpt}{\sharp_{st}}
\newcommand{\symm}{\mathfrak S_3}
\newcommand{\etset}{G\text{-\'Et}}
\DeclareMathOperator{\diag}{\mathrm{diag}}
\DeclareMathOperator{\age}{\mathrm{age}}

\title{Mass formulas and stringy point-count for semi-direct products of tame abelian groups and wild symmetric or cyclic groups in characteristic three}
\author{Takahiro Yamamoto\thanks{Affiliation: Osaka City University Advanced Mathematical Institute (3-3-138, Sugimoto, Sumiyoshi-Ku, Osaka city, Osaka-fu), 

Department of Mathematics, Graduate School of Science, Osaka University (Toyonaka, Osaka 560-0043)

email:mass.11235813@gmail.com}}
\date{}

\begin{document}

\maketitle

\begin{abstract}
    In positive characteristic, there exist counterexamples to the statement corresponding to Batyrev's theorem concerning the McKay correspondence. In this paper, we give another computation of the counterexamples by using stringy-point count for some sequences of quotient varieties. There is a proof of the Serre-Bhagava's mass formula, which is important formula in the number theory, from the computation of stringy-point count of a quotient variety associated to a representation of symmetric group. Our computation of the stringy-point count in this paper is regarded as a variation of this proof to give different versions of mass formula.
\end{abstract}

\tableofcontents
\section{Introduction}
Our study begins with the following theorem, which is a version of the McKay correspondence.
\begin{thm}[\cite{Bat}]
Let \(G\subset\SL_d(\CC)\) be a finite subgroup. Suppose there exists a crepant resolution \(Y\to X\) of a quotient variety \(X=\AA_{\CC}^d/G\). Then the topological Euler characteristic \(e(Y)\) of \(Y\) and the number of conjugacy classes of \(G\) is equal:
\[
e(Y)=\sharp\mathrm{Conj}(G)
\]
\end{thm}
It is natural to ask if the same equality holds in positive characteristic when replacing the topological Euler characteristic by the \(l\)-adic Euler characteristic. The \(l\)-adic Euler characteristic is defined by the alternating sum of the dimensions of the \(l\)-adic \'etale cohomology with compact support for a prime \(l\) different from the characteristic of the base field. Because the \(l\)-adic Euler characteristic is equal to the topological Euler characteristic in characteristic zero, we get a natural extension of the assertion of Batyrev's theorem to arbitrary characteristic.

In positive characteristic, it is known that the equation of Batyrev's theorem holds for the quotient variety of the \(2n\)-dimensional representation defined as a direct sum of two standard representations of the \(n\)-th symmetric group. This follows from Kedlaya's formula \cite[Theorem 1.1]{Ked}, see \cite[Theorem 5.21]{WY}. In \cite{Yas4}, Yasuda studies quotient varieties for \(p\)-cyclic groups in characteristic \(p\). Especially, he proves that a quotient variety \(\AA^3/C_p\) has a crepant resolution if \(p=3\), and the equation of Batyrev's theorem holds. If \(p\geq 5\), the quotient variety does not have crepant resolutions. (\cite[Theorem 1.1, Corollary 6.25]{Yas4})

In the paper \cite{Yam}, we get examples of quotient varieties which have crepant resolutions and the Euler characteristic of the resolutions are not equal to the number of conjugacy classes in characteristic three. The example is given by concrete computation of crepant resolutions and its Euler characteristics. In this paper, we give another computation of the Euler characteristic of the crepant resolutions for some cases treated in \cite{Yam} using stringy-point count invariant. 

The stringy-point count equals a weighted count of \(G\)-\'etale algebras over a local field. In \cite{Yam}, we studied quotient varieties associated to representations of semi-direct products of tame abelian group and wild symmetric or cyclic group. For some of these cases, we compute the weighted count as follows.
\begin{thm}[Main theorem]\label{MTinIntro}
Let \(k\) be a finite field of order \(q\) and characteristic three. Let \(G\subset\SL_3(k)\) be a finite group without pseudo-reflections, that are \(M\in\SL_d(k)\) such that \(\dim_k(k^d)^M=d-1\), and let \(V\) be the associated 3-dimensional representation of \(G\) over \(k\). Let \(l\ne 3\) be a prime.

Suppose that \(l\) satisfies \(q-1\in l\ZZ\). Then we get the following formulas:
\[
\sum_{A\in\etset(K)}\frac{q^{3-\bsym v_V(A)}}{C_G(H_A)}=\left\{\begin{array}{cc}
q^3+(2+\frac{l-1}6)q^2+\frac{l-1}6q&(G\cong C_l\rtimes C_3)\\
q^3+(2+\frac{(l-1)(l+4)}6)q^2+\frac{(l-1)(l-2)}6q&(G\cong C_l^2\rtimes C_3)
\end{array}\right.,
\] 
where \(\etset(K)\) is the set of \(G\)-\'etale \(K\)-algebras, \(\bsym v_V\) is the \(v\)-function associated \(V\), \(H_A\) is a stable subgroup of a summand of \(A\), and \(C_G(H_A)\) is the centralized group of \(H_A\) in \(G\). For more detail, see section three.

On the other hand, when we suppose that \(l\) satisfies \(q-1\in 2l\ZZ\), we get  the following formulas:
\[
\sum_{A\in\etset(K)}\frac{q^{3-\bsym v_V(A)}}{C_G(H_A)}=\left\{\begin{array}{cc}
q^3+6q^2+q&(G\cong C_2^2\rtimes\symm)\\
q^2+\frac{(l+5)(l+7)}{12}q^2+\frac{(l+1)(l+5)}{12}q&(l\ne 2,\ G\cong C_l^2\rtimes\symm)
\end{array}\right.
\]

\end{thm}

The \textit{Serre-Bhargava mass formula} is represented as
\[
\sum_{A\in S_n\text{-\'Et}(K)}q^{-\bsym a(A)}=\sum_{m=0}^{n-1}P(n,n-m)q^{-m},
\]
where \(\bsym a\) is the Artin conductor of induced representation \(\rho:G_K\to\GL_n(\CC)\) and \(P(n,j)\) is the number of partitions of \(n\) into \(j\) positive integers (see \cite{Bha}). In \cite{WY}, another proof of the formula was given, which uses the stringy-point count of a quotient variety by a symmetric group. The left-hand-side of the mass formula is a weighted count of Galois extensions. Thus the formulas in the main theorem can be regarded as new versions of mass formula.

We prove this theorem by concrete computation. Firstly, we classify the \(G\)-\'etale algebras by the corresponding field extension \(L/k((t))\). For each classes, we compute the value of \(v\)-function \(\bsym v_V\) by compute a basis of the tuning module. We compute the number of \(G\)-\'etale algebras with a prescribed value of \(\bsym v_V\). Then we can compute the sum in the left-hand-side in the formulas. 

An important point of the proof is the computation of \(v\)-functions. The \(v\)-function corresponds to a local hight of some stack, for instance see \cite[Section 3.1]{Ell}. The module resolvent defined by Fr\" ohlich in \cite{Fro} also coincides with \(v\)-function. The \(v\)-function coincides with age of an element of group if the order of the group is not divided by the characteristic of base field. But, if the order of the group is divided by the characteristic, computing the \(v\)-function is difficult. The \(v\)-function is computed in the following cases.
\begin{itemize}
    \item In \cite[Theorem 4.8]{WY}, the \(v\)-function is computed for permutation representations.
    \item In \cite[Corollary 11.4]{Yas3}, the \(v\)-function is computed for the stable hyperplane in permutation representations.
    \item The \(v\)-function is computed for the representation of \(p\)-cyclic group in \cite[Proposition 6.9]{Yas4}, and for the representation of \(p^r\)-cyclic group in \cite[Theorem 3.11]{TY} as generalization.
    \item In \cite[Theorem 5.3]{Yam2}, the \(v\)-function is computed for the \(2\)-dimensional representation of the group \((\ZZ/p\ZZ)^2\). This is the first case that \(v\)-function is not determined by ramification filtration.
\end{itemize}
In the proof of the main theorem of this paper, we compute the \(v\)-function for the representation of symmetric group of degree three in characteristic three.

The stringy-point count gives the number of \(k\)-point of a crepant resolution. By using the Weil conjecture for \(l\)-adic cohomology, we can get the Euler characteristic of crepant resolution from the stringy-point count. Thus we get the following corollary from the main theorem.
\begin{coro}
Let \(G\) be a finite subgroup of \(\SL_3(k)\) without pseudo-reflections. For a crepant resolution \(Y\to \AA^3_k/G\), we have
\[
\chi(Y)=\begin{cases}
3+\frac{l-1}3&(G\cong\ZZ/l\ZZ\rtimes\ZZ/3\ZZ)\\
3+\frac{l^2-1}3&(G\cong(\ZZ/l\ZZ)^2\rtimes\ZZ/3\ZZ)\\
\frac{(l-1)(l-2)}6+2l+4&(G\cong(\ZZ/l\ZZ)^2\rtimes\symm)
\end{cases}
\]
\end{coro}

The corollary supports the result of \cite{Yam} for the Euler characteristic. 

The outline of the paper is as follows. In section two, we explain notations and conventions of this paper. In section two, we give a some notations and conventions used through this paper. In section three, we give several facts for the stringy-point count or the stringy-motive, which is an invariant regarded as a generalization of the stringy-point count. In section four, we prove the main theorem by a case-by-case analysis. The key of the proof is the computation of \(v\)-function especially for symmetric extension. The computation of \(v\)-function is generally difficult, but we can in this cases.
\subsection*{Acknowledgements}
I am grateful to Takehiko Yasuda for teaching me and for his valuable comments. This work was partially supported by JSPS KAKENHI Grant Numbers JP18H01112, JP21H04994.
%---------------------------------------------------------------------------------

\section{Notation and Convention}
Throughout this paper, \(k\) denotes the finite field of order \(q\), where \(q\) is a power of a prime \(p\), and \(K\) denotes the Laurent power series field \(k((t))\). The valuation ring of \(K\) with respect to the normalized valuation \(v_K\) is denoted by \(\mathcal O_K\). Note that \(\mathcal O_K\) is the ring \(k[[t]]\) of formal power series. For any field \(F\) of characteristic \(p\), we define
\[
\wp:F\ni x\mapsto x^p-x\in F.
\]
Note that \(\wp\) is additive because the characteristic of \(F\) is \(p\).
 
The cyclic group \(\ZZ/n\ZZ\) of order \(n\) is denoted by \(C_n\) for simplicity. The symmetric group of degree three is denoted by \(\symm\). A \textit{tame} group means a finite group whose order is not divided by the characteristic of the base field. Finite groups which are not tame are called \textit{wild} groups.

For a subgroup \(G\) of \(\SL_d(k)\), we say \(G\) is \textit{small} if \(G\) has no pseudo-reflections, equivalently any element \(g\) of \(G\) satysfies \(\dim_k(k^d)^g\ne 1\). We say \(G\) is \textit{of cyclic type} if \(G\) isomorphic to a semidirect product \(H\rtimes C_3\) where \(H\) is a tame abelian group. If \(G\) isomorphic to a semidirect product \(G\cong H\rtimes\symm\) where \(H\) is a tame abelian group, we say \(G\) is \textit{of symmetric type}.

%-------------------------------------------------------------------------------

\section{Stringy motive and stringy-point count}
In this section, we give some facts of the stringy-point count. For an \(\mathcal O_K\)-variety \(X\), that means \(X\) is a reduced \(\QQ\)-Gorenstein quasi-projective \(\mathcal O_K\)-scheme such that flat, of finite type, equi-dimensional over \(\mathcal O_K\), and the structure morphism \(X\to\Spec\mathcal O_K\) is smooth on an open dence subscheme of \(X\), we denote its \textit{stringy point-count} by \(\stpt X\), which is an invariant defined as the volume of \(X\) with respect to a certain \(p\)-adic measure. For more detail see \cite{Yas}. A crepant resolution means a proper birational morphism \(Y\to X\) such that \(Y\) is smooth over \(\mathcal O_K\) and the relative canonical divisor \(K_{Y/X}\) is zero. For a crepant resolution \(Y\to X\), we have
\[
\sharp Y(k)=\sharp_{st}X.
\]
Thus we can use the stringy point-count for counting the \(k\)-points on a crepant resolution.

For a finite group \(G\), a finite \'etale \(K\)-algebras \(A\) of degree \(\sharp G\) endowed with a \(G\)-action and satisfying \(A^G=K\) is called a \(G\)\textit{-\'etale \(K\)-algebra}. An \textit{isomorphism} of \(G\)-\'etale \(K\)-algebras is a \(G\)-equivaliant isomorphisms of \(K\)-algebras. We denote the set of the isomorphic classes of \(G\)-\'etale \(K\)-algebras by \(\etset(K)\). Note that each \(A\in\etset(K)\) is written as
\[
A=L_1\oplus L_2\oplus\cdots\oplus L_m,
\]
where \(\{L_i\}_{i=1}^m\) are copies of a Galois extension \(L/K\). Let \(H_A\subset G\) be the stabilizer of \(L_1\). Then \(H_A\) is isomorphic to the Galois group of \(L/K\). 

Let \(V=\AA_{\mathcal O_K}^n\) be an \(n\)-dimensional linear representation of \(G\) and let \(\mathcal O_K[x_1,\ldots,x_n]\) be the coordinate ring of \(V\). Then the linear part \(M=\sum_{i=1}^n\mathcal O_Kx_i\) of \(\mathcal O_K[x_1,\ldots,x_n]\) is the free \(\mathcal O_K\)-module of degree \(n\) endowed with \(G\)-action.

For \(A\in\etset(K)\), we define the \textit{tuning module} \(\Xi_{A/K}^V\) of \(A\) by
\[
\Xi_{A/K}^V=\Hom_{\mathcal O_K}^G(M,\mathcal O_A),
\]
where \(\mathcal O_A\) is the integer ring of \(A\) over \(\mathcal O_K\) and \(\Hom_{\mathcal O_K}^G(M,\mathcal O_A)\) is the set of all the \(G\)-equivariant \(\mathcal O_K\)-linear morphisms \(M\to\mathcal O_A\).
Then the \(v\)-function \(\bsym v_V\) on \(\etset(K)\) is defined by
\[
\bsym v_V(A)=\frac 1{\sharp G}\mathrm{length}_{\mathcal O_K}\frac{\Hom_{\mathcal O_K}(M,\mathcal O_A)}{\mathcal O_A\cdot\Xi_{A/K}^V}.
\]
The \(v\)-function satisfies that
\[
\bsym v_V(A)=\bsym v_{V|_{H_A}}(L),
\]
where \(V|_{H_A}\) is a representation given by regarding \(V\) as a representation of \(H_A\).

Note that, the \(v\)-function corresponds to the \textit{age} of an element of \(H_A\) when \(H_A\) is tame. For \(h\in\GL_d(K)\) of order \(l\) which is not divided by \(p\), the eigenvalues of \(h\) is written as \(\zeta^{a_1},\zeta^{a_2},\ldots,\zeta^{a_d}\), where \(\zeta\in K\) is the primitive \(l\)-th root of unity and \(0\leq a_i<l\) for any \(i\). Then the age of \(h\) is defined by
\[
\age(h)=\frac 1l\sum_{i=1}^da_i.
\]
\begin{lemm}[\cite{WY}, Lemma 4.3]\label{lem:tameVFct}
Take a \(G\)-\'etale \(K\)-algebra \(A=L^{\oplus n}\). Suppose that \(H_A\) is tame and the ramification index of \(L/K\) is \(e\). Let \(\zeta\in K\) be a primitive \(e\)-th root of unity. Let \(g\in H_A\) be an element which acts on \(L\) by
\[
g(\pi)=\zeta\pi
\]
for a uniformizer \(\pi\) of \(L\). Then
\[
\bsym v_V(A)=\age(g).
\]
\end{lemm}

The stringy-point count is equal to a weighted sum of \(G\)-\'etale \(K\)-algebras.
\begin{thm}[{\cite[Theorem 7.4, Proposition 8.5]{Yas}}]\label{strPtCt}
Let \(\sharp k=q\). If \(V=\AA^n_{\mathcal O_K}\) is a linear representation of a finite group \(G\). If \(V/G\) is \'etale in codimension one, then we have 
\[
\stpt (V/G)=\sum_{A\in\etset(K)}\frac{q^{n-\bsym v_V(A)}}{\sharp C_G(H_A)}.
\]
\end{thm}
This is one of realizations of the following equation called the \textit{wild McKay correspondence}. The wild McKay correspondence is a equation between \textit{stringy motive} and an integration on a moduli space of \(G\)-torsors over \(\Spec k((t))\), denoted by \(\Delta_G\). The stringy motive is defined as a motivic integration on arc space. For a variety \(X\), we denote the stringy motive of \(X\) by \(M_{st}(X)\).
\begin{thm}[{\cite[Corollary 14.4]{Yas2}}]
Let \(V,G\) be same as in Theorem \ref{strPtCt}. Let \(X\) be a quotient variety \(X=V/G\). Then we have
\[
M_{st}(X)=\int_{\Delta_G}\mathbb L^{d-\bsym v_V}.
\]
\end{thm}

This is generalization of the motivic McKay correspondence in characteristic zero proved by Batyref \cite{Bat} and Denef-Loeser \cite{DL}.
%---------------------------------------------------------------------------------

\section{Main theorem}
In this section, we set \(p=3\). Namely, the base field \(k\) is a finite field of order \(q=3^r\). We give a proof of the following theorem.
\begin{thm}[Main theorem]\label{MT}
Let \(G\subset\SL_3(k)\) be a small finite group and let \(V\) be the natural representation of \(G\). Let \(l\ne 3\) be a prime.

Suppose that \(l\) satisfies \(q-1\in l\ZZ\). Then we get the following formulas for the case that \(G\) is of cyclic type:
\[
\sum_{A\in\etset(K)}\frac{q^{3-\bsym v_V(A)}}{\sharp C_G(H_A)}=\left\{\begin{array}{cc}
q^3+(2+\frac{l-1}6)q^2+\frac{l-1}6q&(G\cong C_l\rtimes C_3)\\
q^3+(2+\frac{(l-1)(l+4)}6)q^2+\frac{(l-1)(l-2)}6q&(G\cong C_l^2\rtimes C_3)
\end{array}\right.
\] 

On the other hand, when we suppose that \(l\) satisfies \(q-1\in 2l\ZZ\), we get  the following formulas for the case that \(G\) is of symmetric type:
\[
\sum_{A\in\etset(K)}\frac{q^{3-\bsym v_V(A)}}{\sharp C_G(H_A)}=\left\{\begin{array}{cc}
q^3+6q^2+q&(G\cong C_2^2\rtimes\symm)\\
q^2+\frac{(l+5)(l+7)}{12}q^2+\frac{(l+1)(l+5)}{12}q&(l\ne 2,\ G\cong C_l^2\rtimes\symm)
\end{array}\right.
\]
\end{thm}

%----------------------------------------------------------------------------

\subsection{The case that \textit{G} is cyclic type}
Firstly, we consider the case that \(G\) is a small subgroup of \(\SL_3(k)\) isomorphic to \(C_l\rtimes C_3\) where \(l\) is a prime different from three. From \cite[Lemma 3.1]{Yam}, we may assume that any elements of \(C_l\subset G\) are diagonalized and \(C_3\subset G\) is generated by
\[
\bsym S=\begin{bmatrix}
0&1&0\\
0&0&1\\
1&0&0
\end{bmatrix}.
\]

In this case, the set \(\etset(K)\) is divided into three parts:
\begin{gather*}
S_1=\{A\in\etset(K)\mid H_A\text{ is trivial}\},\\
S_2=\{A\in\etset(K)\mid H_A\cong C_3\},\\
S_3=\{A\in\etset(K)\mid H_A\cong C_l\},
\end{gather*}
where \(H_A\) denotes the stable subset of a connected component of \(A\) as in previous section.

If \(A\in S_1\), then \(A=K^{\oplus 3l}\). As \(A\) is unramified, we have \(\bsym v_V(A)=0\) from \cite[Lemma 3.4]{WY}. Obviously, \(C_G(H_A)=G\). Thus we get
\begin{align}
    \sum_{A\in S_1}\frac{q^{3-\bsym v_V(A)}}{\sharp C_G(H_A)}=\frac 1{3l}q^3.\label{eq:4.1.1}
\end{align}

If \(A\in S_2\), then \(A=L^{\oplus l}\) where \(L\) is a \(C_3\)-extension over \(K\). By the Artin-Schreier theory, there exist \(a\in K/\wp(K)\) such that \(K[x]/(x^3-x-a)\cong L\). This implies that there exists \(\alpha\in L\) uniquely such that \(\alpha\) generates \(L\) as \(K\)-algebra, \(\alpha\) satisfies
\[
\alpha^3-\alpha\in k/\wp(k)\oplus \bigoplus_{3\nmid j>0}kt^{-j},
\]
and \(C_3\) acts on \(L\) by \(\alpha\cdot\bsym S=\alpha+1\). Here, we fix a representative for each element of the quotient \(k/\wp (k)\), and we denote the set of the representatives also by \(k/\wp(k)\). Hence the set \(k/\wp(k)\oplus \bigoplus_{3\nmid j>0}kt^{-j}\) is a set of representatives of \(K/\wp(K)\). We set \(a=\alpha^3-\alpha\). Let \(j=-v_K(a)\). If \(j=0\) then \(A\) is unramified and we get \(\bsym v_V(A)=0\) from \cite[Lemma 3.4]{WY}. When \(j>0\), from \cite[Lemma 11.1]{Yas3} and \cite[Proposition 7 in page 50]{Ser}, we get
\[
\bsym v_V(A)=\frac{d_{L/K}}2=j+1.
\]
We set
\[
S_{2,m}=\{A\in S_2\mid\bsym v_V(A)=m\}.
\]
If \(3\nmid (m-1)>0\), then we have
\[
\sharp S_{2,0}=\sharp((k/\wp(k))\backslash\{0\})=2,
\]
and
\begin{align*}
\sharp S_{2,m}&=\sharp\left(k/\wp(k)\oplus\bigoplus_{0<j<m-1, 3\nmid j}kt^{-j}\oplus k^\times t^{-m}\right)\\
&=3q^{m-\floor{\frac{m-1}3}-2}(q-1).
\end{align*}
As \(l\) is not divided by three, \(C_G(C_3)=C_3\). Therefore, we get

\begin{align}
    \sum_{A\in S_2}\frac{q^{3-\bsym v_V(A)}}{\sharp C_G(H_A)}
    &=\sum_{m=0}^\infty\sum_{A\in S_{2,m}}\frac{q^{3-m}}{3}\notag\\
    &=\frac 23q^3+\frac{q^3}3\sum_{3\nmid m-1>0}\sharp S_{2,m}q^{-m}\notag\\
    &=\frac 23q^3+q(q-1)\sum_{3\nmid m-1>0}q^{-\floor{\frac{m-1}3}}\notag\\
    &=\frac 23q^3+q(q-1)\sum_{r=0}^\infty( q^{-\floor{\frac{3r}3}}+q^{-\floor{\frac{3r+1}3}})\notag\\
    &=\frac 23q^3+2q(q-1)\sum_{r=0}^\infty q^{-r}\notag\\
    &=\frac 23q^3+2q(q-1)\frac q{q-1}\notag\\
    &=\frac 23q^3+2q^2.\label{eq:4.1.2}
\end{align}
If \(A\in S_3\), \(A=L^{\oplus 3}\) where \(L\) is a \(C_l\)-extension of \(K\). From the assumption \(l\mid (q-1)\), \(K\) has all the \(l\)-th roots of unity. A quotient group \(K^\times/(K^\times)^l\) is isomorphic to a bicyclic group \(C_l^2\). We fix a generator \(\mu\) of \(k^\times\). Then \(K^\times/(K^\times)^l\) is generated by classes of \(\mu\) and \(t\). By the Kummer theory, there exists \(\alpha\in L\) such that \(L\) is generated by \(\alpha\) as \(K\)-algebra and
\[
\alpha^l=\mu\text{ or }\mu^it\quad(0\leq i\leq l-1).
\]
If \(\alpha^l=\mu\) then \(A\) is unramified. Thus \(\bsym v_V(A)=0\) from \cite[Lemma 3.4]{WY}. We suppose that \(\alpha^l=\mu^it\). Let \(\zeta_l\in k\) is a primitive \(l\)-th root of unitiy. We choose a generator \(h\) of \(C_l\) such that \(\alpha\cdot h=\zeta_l\alpha\). From Lamma \ref{lem:tameVFct}, we get
\[
\bsym v_V(A)=\age(h).
\]

We put
\[
S_{3,m}=\{A\in S_3\mid \bsym v_V(A)=m\}.
\]
As the action on \(A\) is defined by choice of a conjugate class represented by \(h\), 
\[
\sharp S_{3,m}=\begin{cases}
\frac{l-1}3&(m=0)\\
\frac{l(l-1)}6&(m=1)\\
\frac{l(l-1)}6&(m=2)
\end{cases}.
\]
In this case, \(C_G(C_l)=C_l\). Therefore, we get
\begin{align}
    \sum_{A\in S_3}\frac{q^{3-\bsym v_V(A)}}{\sharp C_G(H_A)}
    &=\sum_{m=0}^2\sum_{A\in S_{3,m}}\frac{q^{3-m}}{l}\notag\\
    &=\frac{l-1}{3l}q^3+\frac{l-1}6q^2+\frac{l-1}6q.\label{eq:4.1.3}
\end{align}
\begin{thm}\label{thm:4.3}
Let \(G\) be a small subgroup of \(\SL_3(k)\) isomorphic to \(C_l\rtimes C_3\) where \(l\) is a prime different from three. Suppose that \(q-1\in l\ZZ\). For the natural representation \(V\) of \(G\), we get
\[
\sum_{A\in\etset(K)}\frac{q^{3-\bsym v_V(A)}}{\sharp C_G(H_A)}=q^3+\left(2+\frac{l-1}6\right)q^2+\frac{l-1}6q.
\]
\end{thm}
\begin{proof}
By formulas (\ref{eq:4.1.1}), (\ref{eq:4.1.2}), and (\ref{eq:4.1.3}), we get
\begin{align*}
    \sum_{A\in\etset(K)}\frac{q^{3-\bsym v_V(A)}}{\sharp C_G(H_A)}&=\sum_{i=1}^3\sum_{A\in S_i}\frac{q^{3-\bsym v_V(A)}}{\sharp C_G(H_A)}\\
    &=\frac{q^3}{3l}+\left(\frac 23q^3+2q^2\right)+\left(\frac{l-1}{3l}q^3+\frac{l-1}6q^2+\frac{l-1}6q\right)\\
    &=q^3+\left(2+\frac{l-1}6\right)q^2+\frac{l-1}6q.
\end{align*}
\end{proof}

Next, we consider the case that \(G\) is isomorphic to \(C_l^2\rtimes C_3\). As previous case, we may assume that any elements of \(C_l^2\subset G\) are diagonalized and \(C_3\subset G\) is generated by
\[
\bsym S=\begin{bmatrix}
0&1&0\\
0&0&1\\
1&0&0
\end{bmatrix}.
\]
Note that in this case, we have
\[
C_l^2=\{\diag(\zeta_l^a,\zeta_l^b,\zeta_l^c)\mid 0\leq a,b,c\leq l-1,a+b+c\in l\ZZ\}.
\]

In this case, the set \(\etset(K)\) is divided into four parts:
\begin{gather*}
S_1=\{A\mid H_A\text{ is trivial}\},\\
S_2=\{A\mid H_A\cong C_3\},\\
S_3=\{A\mid H_A\cong C_l\},\\
S_4=\{A\mid H_A\cong C_l^2\}.
\end{gather*}

If \(A\in S_1\), \(A\) is unramified. Therefore we get \(\bsym v_V(A)=0\) from \cite[Lemma 3.4]{WY}. As \(C_G(H_A)=G\), we get
\begin{align}
\sum_{A\in S_1}\frac{q^{3-\bsym v_V(A)}}{\sharp C_G(H_A)}=\frac 1{3l^2}q^3.\label{eq:4.1.4}
\end{align}

If \(A\in S_2\), then \(A=L^{\oplus l^2}\) where \(L\) is a \(C_3\)-extension over \(K\). Then we repeat the same argument as in the case \(A\in S_2\) when \(G\cong C_l\rtimes C_3\). We let \(\alpha, a, j\) as in the case. Then we have
\[
\bsym v_V(A)=\begin{cases}
0&(j=0)\\
j+1&(j>0)
\end{cases}.
\]
We put
\[
S_{2,m}=\{A\in S_2\mid\bsym v_V(A)=m\}.
\]
Then
\[
\sharp S_{2,m}=\begin{cases}
2&(m=0)\\
3q^{m-\floor{\frac{m-1}3}-2}(q-1)&(3\nmid(m-1)>0)\\
0&(\text{otherwise})
\end{cases}.
\]
As \(C_G(H_A)=C_3\), we get
\begin{align}
    \sum_{A\in S_2}\frac{q^{3-\bsym v_V(A)}}{\sharp C_G(H_A)}&=\sum_{m=0}^\infty\sum_{A\in S_2}\frac{q^{3-m}}{3}\notag\\
    &=\frac 23q^3+q(q-1)\sum_{3\nmid(m-1)>0}q^{-\floor{\frac{m-1}3}}\notag\\
    &=\frac 23q^3+2q^2.\label{eq:4.1.5}
\end{align}

If \(A\in S_3\), \(A=L^{\oplus 3}\) where \(L\) is a \(C_l\)-extension of \(K\). Similarly as in the case \(A\in S_3\) when \(G\cong C_l\rtimes C_3\), we can choose a generator \(\alpha\in L\) over \(K\) satisfying
\[
\alpha^l=\mu\text{ or }\mu^it\quad(0\leq i\leq l-1).
\]
and \(h\in H_A\) such that \(\alpha\cdot h=\zeta_l\alpha\).

If \(\alpha^l=\mu\) then \(A\) is unramified. Thus \(\bsym v_V(A)=0\) from \cite[Lemma 3.4]{WY}. We suppose that \(\alpha^l=\mu^it\). By Lemma \ref{lem:tameVFct}, we get
\[
\bsym v_V(A)=\age(h).
\]
We put
\[
S_{3,m}=\{A\in S_3\mid \bsym v_V(A)=m\}.
\]
From
\[
C_l^2=\{\diag(\zeta_l^a,\zeta_l^b,\zeta_l^c)\mid 0\leq a,b,c\leq l-1,a+b+c\in l\ZZ\},
\]
we have
\begin{gather*}
    \sharp\{h\in C_l^2\mid\age(h)=1\}=\frac{(l-1)(l+4)}2,\\
    \sharp\{h\in C_l^2\mid\age(h)=2\}=\frac{(l-1)(l-2)}2.\\
\end{gather*}
As the action on \(A\) is defined by choice of a conjugate class represented by \(h\), 
\[
\sharp S_{3,m}=\begin{cases}
\frac{l^2-1}3&(m=0)\\
\frac{l(l-1)(l+4)}6&(m=1)\\
\frac{l(l-1)(l-2)}6&(m=2)
\end{cases}.
\]
In this case, \(C_G(H_A)=C_l^2\). Therefore, we get
\begin{align}
    \sum_{A\in S_3}\frac{q^{3-\bsym v_V(A)}}{\sharp C_G(H_A)}
    &=\sum_{m=0}^2\sum_{A\in S_{3,m}}\frac{q^{3-m}}{l^2}\notag\\
    &=\frac{l^2-1}{3l^2}q^3+\frac{(l-1)(l+4)}{6l}q^2+\frac{(l-1)(l-2)}{6l}q.\label{eq:4.1.6}
\end{align}

If \(A\in S_4\), \(A=L^{\oplus 3}\) where \(L\) is \(C_l^2\)-extension of \(K\). As \(K^\times/(K^\times)^l\) is isomorphic to \(C_l^2\) and it is generated by \(\mu, t\), there are generators \(\alpha,\beta\in L\) such that
\[
\alpha^l=\mu,\ \beta^l=t.
\]
Note that \(\beta\) is an uniformizer of \(L\). We choose \(h_1,h_2\in C_l^2\) satisfying
\begin{align*}
    \alpha\cdot h_1&=\zeta_l\alpha,&\beta\cdot h_1&=\beta,\\
    \alpha\cdot h_2&=\alpha,&\beta\cdot h_2&=\zeta_l\beta.
\end{align*}
From Lemma \ref{lem:tameVFct}, we get
\[
\bsym v_V(A)=\age(h_2).
\]

For pairs of generators \((h_1,h_2),(h^\prime_1,h_2^\prime)\) of \(H_A\), we define a equivalence relation \(\sim\) by
\[
(h_1,h_2)\sim(h^\prime_1,h_2^\prime)\Leftrightarrow\exists g\in G,h_1=gh^\prime_1g^{-1},h_2=gh^\prime_2g^{-1}.
\]
Then a \(G\)-action on \(A\) is defined by the equivalent class of the pair of generators \((h_1,h_2)\) of \(C_l^2\). We have
\begin{gather*}
    \sharp\{(h_1,h_2)\mid h_1,h_2 \text{ generate }C_l^2,\age(h_2)=1\}=\frac{(l-1)(l+4)}2(l^2-l),\\
    \sharp\{(h_1,h_2)\mid h_1,h_2 \text{ generate }C_l^2,\age(h_2)=2\}=\frac{(l-1)(l-2)}2(l^2-l).
\end{gather*}
Hence, for
\[
S_{4,m}=\{A\in S_4\mid \bsym v_V(A)=m\},
\]
we get
\begin{align*}
    \sharp S_{4,m}=\begin{cases}
    \frac{(l-1)(l+4)(l^2-l)}6&(m=1)\\
    \frac{(l-1)(l-2)(l^2-l)}6&(m=2)
    \end{cases}.
\end{align*}
Therefore, as \(C_G(H_A)=C_l^2\), we get
\begin{align}
    \sum_{A\in S_4}\frac{q^{3-\bsym v_V(A)}}{\sharp C_G(H_A)}
    &=\sum_{m=1}^2\sum_{A\in S_{4,m}}\frac{q^{3-m}}{l^2}\notag\\
    &=\frac{(l-1)^2(l+4)}{6l}q^2+\frac{(l-1)^2(l-2)}{6l}q.\label{eq:4.1.7}
\end{align}
\begin{thm}\label{thm:4.5}
Let \(G\) be a small subgroup of \(\SL_3(k)\) isomorphic to \(C_l^2\rtimes C_3\) where \(l\) is a prime different from three. Suppose that \(q-1\in l\ZZ\). For the natural representation \(V\) of \(G\), we get
\[
\sum_{A\in\etset(K)}\frac{q^{3-\bsym v_V(A)}}{\sharp C_G(H_A)}=q^3+\left(2+\frac{(l-1)(l+4)}6\right)q^2+\frac{(l-1)(l-2)}6q
\]
\end{thm}
\begin{proof}
By the formulas (\ref{eq:4.1.4}), (\ref{eq:4.1.5}), (\ref{eq:4.1.6}), and (\ref{eq:4.1.7}), we get
\begin{align*}
    \sum_{A\in\etset(K)}\frac{q^{3-\bsym v_V(A)}}{\sharp C_G(H_A)}&=\sum_{i=1}^4\sum_{A\in S_i}\frac{q^{3-\bsym v_V(A)}}{\sharp C_G(H_A)}\\
    &=\frac{q^3}{3l^2}+\left(\frac 23q^3+2q^2\right)\\
    &\qquad+\left(\frac{(l^2-1)}{3l^2}q^3+\frac{(l-1)(l+4)}{6l}q^2+\frac{(l-1)(l-2)}{6l}q\right)\\
    &\qquad+\left(\frac{(l-1)^2(l+4)}{6l}q^2+\frac{(l-1)^2(l-2)}{6l}q\right)\\
    &=q^3+\left(2+\frac{(l-1)(l+4)}6\right)q^2+\frac{(l-1)(l-2)}6q.
\end{align*}
\end{proof}
%----------------------------------------------------------------------------

\subsection{The case that \textit{G} is symmetric type}
In this section, we consider the case that \(G\) is a small subgroup of \(\SL_3(k)\) isomorphic to \(C_l^2\rtimes\symm\) where \(l\) is a prime different from three. From \cite[Lemma 3.2]{Yam}, we may assume that any elements of \(C_l^2\) are diagonalized and \(\symm\subset G\) is generated by
\[
\bsym S=\begin{bmatrix}
0&1&0\\
0&0&1\\
1&0&0
\end{bmatrix},\ \bsym T=\begin{bmatrix}
0&0&-1\\
0&-1&0\\
-1&0&0
\end{bmatrix}.
\]

Firstly, we consider the case that \(l\ne 2\). In this case, the set \(\etset(K)\) is divided into seven parts:
\begin{gather*}
S_1=\{A\mid H_A\text{ is trivial}\},\\
S_2=\{A\mid H_A\cong C_3\},\\
S_3=\{A\mid H_A\cong C_l\},\\
S_4=\{A\mid H_A\cong C_l^2\},\\
S_5=\{A\mid H_A\cong C_2\},\\
S_6=\{A\mid H_A\cong C_{2l}\},\\
S_7=\{A\mid H_A\cong \symm\}.
\end{gather*}

If \(A\in S_1\) then \(A\) is unramified and \(C_G(H_A)=G\). Thus we get 
\(\bsym v_V(A)=0\) from \cite[Lemma 3.4]{WY} and
\begin{align}\label{eq:4.2.1}
\sum_{A\in S_1}\frac{q^{3-\bsym v_V(A)}}{\sharp C_G(H_A)}=\frac 1{6l^2}q^3.
\end{align}

If \(A\in S_2\), then \(A=L^{\oplus 2l^2}\) where \(L\) is a \(C_3\)-extension over \(K\). Then we discuss as in the case \(A\in S_2\) when \(G\cong C_l\rtimes C_3\). We let \(\alpha, a, j\) as in the case. Then we have
\[
\bsym v_V(A)=\begin{cases}
0&(j=0)\\
j+1&(j>0)
\end{cases}.
\]
We put
\[
S_{2,m}=\{A\in S_2\mid\bsym v_V(A)=m\}.
\]
As \(\bsym S\) and \(\bsym S^2\) are conjugate in \(G\), we have
\[
\sharp S_{2,m}=\begin{cases}
1&(m=0)\\
\frac 32q^{m-\floor{\frac{m-1}3}-2}(q-1)&(3\nmid(m-1)>0)\\
0&(\text{otherwise})
\end{cases}.
\]
As \(C_G(H_A)=C_3\), we get
\begin{align}
    \sum_{A\in S_2}\frac{q^{3-\bsym v_V(A)}}{\sharp C_G(H_A)}&=\sum_{m=0}^\infty\sum_{A\in S_2}\frac{q^{3-m}}{3}\notag\\
    &=\frac 13q^3+\frac 12q(q-1)\sum_{3\nmid(m-1)>0}q^{-\floor{\frac{m-1}3}}\notag\\
    &=\frac 13q^3+q^2.\label{eq:4.2.2}
\end{align}

If \(A\in S_3\), \(A=L^{\oplus 3}\) where \(L\) is a \(C_l\)-extension of \(K\). Similarly as in the case \(A\in S_3\) when \(G\cong C_l\rtimes C_3\), we can choose a generator \(\alpha\in L\) over \(K\) satisfying
\[
\alpha^l=\mu\text{ or }\mu^it\quad(0\leq i\leq l-1),
\]
and \(h\in H_A\) such that \(\alpha\cdot h=\zeta_l\alpha\).

If \(\alpha^l=\mu\) then \(A\) is unramified. Thus \(\bsym v_V(A)=0\) from \cite[Lemma 3.4]{WY}. We suppose that \(\alpha^l=\mu^it\). By Lemma \ref{lem:tameVFct}, we get
\[
\bsym v_V(A)=\age(h).
\]
We put
\[
S_{3,m}=\{A\in S_3\mid \bsym v_V(A)=m\}.
\]
When \(h=\diag(\zeta_l^a,\zeta_l^b,\zeta_l^c)\in C_l^2\), if two of \(a,b,c\) are equal, then
\[
\sharp C_G(H_A)=2l^2,
\]
otherwise
\[
\sharp C_G(H_A)=l^2.
\]
We divide the set \(S_{3,m}\) into two parts:
\begin{gather*}
S_{3,m,f}=\{A\in S_{3,m}\mid\sharp C_G(H_A)=l^2\},\\
S_{3,m}^\prime=\{A\in S_{3,m}\mid\sharp C_G(H_A)=2l^2\}.
\end{gather*}
Now the elements of \(S_{3,m}^\prime\) corresponds to the conjugate class in \(C_l^2\), which is represented by \(\diag(\zeta_l^a,\zeta_l^{2l-2a},\zeta_l^a)\), and choice of exponent \(i\) in \(\alpha^l=\mu^it\) from \(0\leq i\leq l-1\). Moreover, we have
\[
\age(\diag(\zeta_l^a,\zeta_l^{2l-2a},\zeta_l^a))=\begin{cases}
1&(1\leq a\leq \floor{\frac l2})\\
2&(\floor{\frac l2}<a\leq l-1)
\end{cases}.
\]
Hence
\begin{gather*}
    \sharp S_{3,1}^\prime=l\floor{\frac l2},\\
    \sharp S_{3,2}^\prime=l\left(l-1-\floor{\frac l2}\right).
\end{gather*}
Moreover, we have
\begin{gather*}
    \sharp S_{3,1,f}=\frac l6\left(\frac{(l-1)(l+4)}{2}-3\floor{\frac l2}\right),\\
    \sharp S_{3,2,f}=\frac l6\left(\frac{(l-1)(l-2)}{2}-3\left(l-1-\floor{\frac l2}\right)\right).
\end{gather*}
On the other hand, \(S_{3,0}\) is corresponding to the conjugate classes in \(C_l^2\). Hence
\begin{gather*}
    \sharp S_{3,0}^\prime=l-1,\\
\sharp S_{3,0,f}=\frac{l^2-1-3(l-1)}6=\frac{(l-1)(l-2)}6.
\end{gather*}
Therefore, we get
\begin{align}
    \sum_{A\in S_3}\frac{q^{3-\bsym v_V(A)}}{\sharp C_G(H_A)}
    &=\sum_{m=0}^2\left(\sum_{A\in S_{3,m,f}}\frac{q^{3-m}}{l^2}+\sum_{A\in S_{3,m}^\prime}\frac{q^{3-m}}{2l^2}\right)\notag\\
    &=\sum_{m=0}^2\left(\frac{\sharp S_{3,m,f}}{l^2}+\frac{\sharp S_{3,m}^\prime}{2l^2}\right)q^{3-m}\notag\\
    &=\left(\frac{(l-1)(l-2)}{6l^2}+\frac{l-1}{2l^2}\right)q^3\notag\\
    &\qquad+\left(\frac 1{6l}\left(\frac{(l-1)(l+4)}2-3\floor{\frac l2}\right)+\frac 1{2l}\floor{\frac l2}\right)q^2\notag\\
    &\qquad+\left(\frac 1{6l}\left(\frac{(l-1)(l-2)}2-3\left(l-1-\floor{\frac l2}\right)\right)+\frac 1{2l}\left(l-1-\floor{\frac l2}\right)\right)q\notag\\
    &=\frac{l^2-1}{6l^2}q^3+\frac{(l-1)(l+4)}{12l}q^2+\frac{(l-1)(l-2)}{12l}q.\label{eq:4.2.3}
\end{align}

If \(A\in S_4\), \(A=L^{\oplus 6}\) where \(L\) is \(C_l^2\)-extension of \(K\). As \(K^\times/(K^\times)^l\) is isomorphic to \(C_l^2\) and it is generated by \(\mu, t\), there are generators \(\alpha,\beta\in L\) such that
\[
\alpha^l=\mu,\ \beta^l=t.
\]
We choose \(h_1,h_2\in C_l^2\) satisfying
\begin{align*}
    \alpha\cdot h_1&=\zeta_l\alpha,&\beta\cdot h_1&=\beta,\\
    \alpha\cdot h_2&=\alpha,&\beta\cdot h_2&=\zeta_l\beta.
\end{align*}
By Lemma \ref{lem:tameVFct}, we have
\[
\bsym v_V(A)=\age(h_2),
\]
because \(\beta\) is an uniformizer of \(L\).

For pairs of generators \((h_1,h_2),(h^\prime_1,h_2^\prime)\) of \(H_A\), we define a equivalence relation \(\sim\) by
\[
(h_1,h_2)\sim(h^\prime_1,h_2^\prime)\Leftrightarrow\exists g\in G,h_1=gh^\prime_1g^{-1},h_2=gh^\prime_2g^{-1}.
\]
Then a \(G\)-action on \(A\) is defined by the equivalent class of the pair of generators \((h_1,h_2)\) of \(C_l^2\). We have
\begin{gather*}
    \sharp\{(h_1,h_2)\mid h_1,h_2 \text{ generate }C_l^2,\age(h_2)=1\}=\frac{(l-1)(l+4)}2(l^2-l),\\
    \sharp\{(h_1,h_2)\mid h_1,h_2 \text{ generate }C_l^2,\age(h_2)=2\}=\frac{(l-1)(l-2)}2(l^2-l).
\end{gather*}
When \(h_1,h_2\in C_l^2\) generate \(C_l^2\), for \(g\in G\), we have
\begin{align*}
    h_1=gh_1g^{-1},h_2=gh_2g^{-1}&\Leftrightarrow g\in C_G(C_l^2)\\
    &\Leftrightarrow g\in C_l^2.
\end{align*}
Hence, for
\[
S_{4,m}=\{A\in S_4\mid \bsym v_V(A)=m\},
\]
we get
\begin{align*}
    \sharp S_{4,m}=\begin{cases}
    \frac{(l-1)(l+4)(l^2-l)}{12}&(m=1)\\
    \frac{(l-1)(l-2)(l^2-l)}{12}&(m=2)
    \end{cases}.
\end{align*}
As \(C_G(H_A)=C_l^2\), we get
\begin{align}
    \sum_{A\in S_4}\frac{q^{3-\bsym v_V(A)}}{\sharp C_G(H_A)}
    &=\sum_{m=1}^2\sum_{A\in S_{4,m}}\frac{q^{3-m}}{l^2}\notag\\
    &=\frac{(l-1)^2(l+4)}{12l}q^2+\frac{(l-1)^2(l-2)}{12l}q.\label{eq:4.2.4}
\end{align}

If \(A\in S_5\), then \(A\cong L^{\oplus 3l^2}\) where \(L\) is a \(C_2\)-extension of \(K\). From the Kummer theory, \(L\) has a generator \(\alpha\in L\) over \(K\) satisfying
\[
\alpha^2=\mu\text{ or }\mu^it\quad(i=0,1).
\]
where \(\mu\) is a generator of \(k^\times\) because \(K^\times/(K^\times)^2\cong C_2^2\) which is generated by \(\mu\) and \(t\). Replacing conjugate one, We may assume that
\[
\alpha\cdot\tau=-\alpha.
\]
We put
\[
S_{5,m}=\{A\in S_5\mid \bsym v_V(A)=m\}.
\]

If \(\alpha^2=\mu\), then \(A\) is unramified. Therefore \(\bsym v_V(A)=0\) from \cite[Lemma 3.4]{WY}.

We suppose that \(\alpha^2=\mu^it\ (i=0,1)\). Then \(\alpha\) is an uniformizer of \(L\). Hence we get
\[
\bsym v_V(A)=\age(\tau)=1
\]
from Lemma \ref{lem:tameVFct}.

Therefore
\[
\sharp S_{5,m}=\begin{cases}
1&(m=0)\\
2&(m=1)
\end{cases}.
\]
As \(C_G(H_A)\) is generated by \(\diag(\zeta_l,\zeta_l^{l-2},\zeta_l), \tau\), we have
\begin{align}
    \sum_{A\in S_5}\frac{q^{3-\bsym v_V(A)}}{\sharp C_G(H_A)}
    &=\sum_{m=0}^1\sum_{A\in S_{5,m}}\frac{q^{3-m}}{2l}\notag\\
    &=\frac 1{2l}q^3+\frac 1lq^2.\label{eq:4.2.5}
\end{align}

If \(A\in S_6\), \(A=L^{\oplus 3l}\) where \(L\) is \(C_{2l}\)-extension of \(K\). From the assumption \(q-1\in 2l\ZZ\), \(K\) contains all the \(2l\)-th roots of unity. A quotient group \(K^\times/(K^\times)^{2l}\) is isomorphic to a bicyclic group \(C_{2l}^2\) generated by \(\mu\) and \(t\). By the Kummer theory, there exists \(\alpha\in L\) such that \(L\) is generated by \(\alpha\) as \(K\)-algebra and
\[
\alpha^{2l}\in\{\mu\}\cup\{\mu^it\mid0\leq i\leq 2l-1\}\cup\left\{\mu^{2j-1}t^2\middle |1\leq j\leq l\right\}\cup\{\mu^et^l\mid e=1,2\}.
\]
If \(\alpha^{2l}=\mu\) then \(A\) is unramified. Thus \(\bsym v_V(A)=0\) from \cite[Lemma 3.4]{WY}. We suppose that \(\alpha^{2l}\ne\mu\). Let \(\zeta_{2l}\in k\) is a primitive \(2l\)-th root of unitiy. We may assume that \(H_A\) is generated by \(\diag(1,\zeta_l,\zeta_l^{-1})\tau\). There is a generator \(h\in H_A\) such that
\[
\alpha\cdot h=\zeta_{2l}\alpha.
\]
We write \(h\) as
\[
(\diag(1,\zeta_l,\zeta_l^{-1})\tau)^{2i-1}=\diag(\zeta_l^{-i},\zeta_l^{2i-1},\zeta_l^{-i+1})\tau.
\]
by an \(i\) such that \(1\leq i\leq l\) and \(2i-1\ne l\).
\begin{lemm}\label{lem:4.7}
For above \(A\) including the case that \(l=2\), we get
\[
\bsym v_V(A)=\begin{cases}
1&(v_K(\alpha^{2l})\leq 2, r\geq\frac l2 )\\
2&(v_K(\alpha^{2l})\leq 2, r<\frac l2)\\
1&(v_K(\alpha^{2l})=l)
\end{cases},
\]
where \(r\) is an integer such that \(1\leq r\leq l-1\) and \(2i-1\equiv r\mod l\).
\end{lemm}
\begin{proof}
For each case, the \(\alpha\in L\) is an uniformizer of \(L\). The ramification index of \(L/K\) is equal to \(2l/v_K(\alpha^{2l})\). Hence we get
\[
\bsym v_V(A)=\begin{cases}
\age(h)&(v_K(\alpha^{2l})=1)\\
\age(h^2)&(v_K(\alpha^{2l})=2)\\
\age(h^l)&(v_K(\alpha^{2l}=l)
\end{cases}
\]
from Lemma \ref{lem:tameVFct}.

We consider the case that \(v_K(\alpha^{2l})=1\). The eigenvalues of \(h\) are \(-\zeta_l^r,\pm\zeta_{2l}^{-r}\). Thus we get
\begin{align*}
\age(h)&=\begin{cases}
\frac 1{2l}\{(2r+l)+(2l-r)+(l-r)\}&(1\leq r<\frac l2)\\
\frac 1{2l}\{(2r-l)+(2l-r)+(l-r)\}&(r\geq\frac l2)
\end{cases}\\
&=\begin{cases}
2&(1\leq r<\frac l2)\\
1&(r\geq\frac l2)
\end{cases}.
\end{align*}
Hence
\[
\bsym v_V(A)=\begin{cases}
2&(1\leq r<\frac l2)\\
1&(r\geq\frac l2)
\end{cases}.
\]

Next, we consider the case that \(v_K(\alpha^{2l})=2\). The eigenvalues of \(h^2\) are \(\zeta_l^{2r},\zeta_l^{-r},\zeta_l^{-r}\). Thus we get
\begin{align*}
\age(h^2)&=\begin{cases}
\frac 1l\{2r+(l-r)+(l-r)\}&(1\leq r<\frac l2)\\
\frac 1l\{(2r-l)+(l-r)+(l-r)\}&(r\geq\frac l2)
\end{cases}\\
&=\begin{cases}
2&(1\leq r<\frac l2)\\
1&(r\geq\frac l2)
\end{cases}.
\end{align*}
Hence
\[
\bsym v_V(A)=\begin{cases}
2&(1\leq r<\frac l2)\\
1&(r\geq\frac l2)
\end{cases}.
\]

Finally, we consider the case that \(v_K(\alpha^{2l})=l\). The eigenvalues of \(h^l\) are \(1,-1,-1\). Thus we get
\begin{align*}
\bsym v_V(A)=\age(h^l)&=\frac 12(0+1+1)=1.
\end{align*}
\end{proof}
We put
\[
S_{6,m}=\{A\in S_3\mid \bsym v_V(A)=m\}.
\]
Then we get 
\[
\sharp S_{6,m}=\begin{cases}
l-1&(m=0)\\
\frac{3l(l-1)}2+2(l-1)&(m=1)\\
\frac{3l(l-1)}2&(m=2)
\end{cases}.
\]
In this case, \(C_G(H_A)=H_A\). Therefore, we get
\begin{align}
    \sum_{A\in S_6}\frac{q^{3-\bsym v_V(A)}}{\sharp C_G(H_A)}
    &=\sum_{m=0}^2\sum_{A\in S_{6,m}}\frac{q^{3-m}}{2l}\notag\\
    &=\frac{l-1}{2l}q^3+\frac{(l-1)(3l+4)}{4l}q^2+\frac{3(l-1)}4q.\label{eq:4.2.6}
\end{align}

If \(A\in S_7\) then \(A\cong L^{\oplus l^2}\) where \(L\) is a \(\symm\)-extension over \(K\). We may assume that \(H_A\) is generated by \(\bsym S,\bsym T\). Let \(Q\) be an intermediate field \(L^{\bsym S}\) of the extension \(L/K\). Then the extension \(L/Q\) is a \(C_3\)-extesion and \(Q/K\) is a \(C_2\)-extension. From the Kummer theory, there exists a generator \(\alpha\in Q\) over \(K\) satisfying
\[
\alpha^2=\mu\text{ or }\mu^it\quad(i=0,1).
\]
From the Artin-Schreier theory, there exists a generator \(\beta\) of \(L\) over \(Q\) satisfying
\[
\beta^3-\beta\in\FF_3\lambda\oplus\bigoplus_{3\nmid j>0}k_Q\pi_Q^{-j}.
\]
where \(k_Q\) is the residue field of \(Q\), \(\lambda\in k_Q-\wp(k_Q)\), and \(\pi_Q\) is a uniformaizer of \(Q\). We set
\[
RP_Q=\FF_3\lambda\oplus\bigoplus_{3\nmid j>0}k_Q\pi_Q^{-j}.
\]
As \(\pi_Q\cdot\bsym T=\pm\pi_Q\), the set \(RP_Q\) is invariant under the action of \(\bsym T\). We may assume that
\begin{gather*}
    \alpha\cdot\bsym S=\alpha,\\
    \alpha\cdot\bsym T=-\alpha,\\
    \beta\cdot\bsym S=\beta+1.
\end{gather*}
We write as \(\beta\cdot\bsym T=\sum_{i=0}^2c_i\beta^i\) with \(c_i\in Q\). As \(\bsym T\bsym S=\bsym S^2\bsym T\), we have
\[
\beta\cdot(\bsym T\bsym S)=\beta\cdot(\bsym S^2\bsym T)
\]
and that implies
\[
(c_0+c_1+c_2)+(c_1-c_2)\beta+c_2\beta^2=(c_0-1)+c_1\beta+c_2\beta.
\]
Hence \(c_2=0,c_1=-1\), in other words, \(\beta\cdot\bsym T=c_0-\beta\). As \(\bsym T^2=1\), 
\[
\beta=(\beta\cdot\bsym T)\cdot\bsym T=(c_0\cdot\bsym T-c_0)+\beta.
\]
Hence \(c_0\in Q^{\bsym T}=K\). Let \(b=\beta^3-\beta\in RP_Q\). Then
\[
b\cdot\bsym T=(\beta^3-\beta)\cdot T=(c_0^3-c_0)-b.
\]
Thus \(b\cdot\bsym T+b=c_0^3-c_0\in\wp(K)\). As the set \(RP_Q\) is closed under the \(\bsym T\)-action and addition, \(b\cdot\bsym T+b\in RP_Q\). Therefore \(b\cdot T+b=0\) because \(RP_Q\cap\wp(K)=\{0\}\). As \(c_0^3-c_0=0\), \(c_0\in\FF_3\). By replacing \(\beta\) by \(\beta\pm 1\) if neccesary, we may assume that \(\beta\cdot\bsym T=-\beta\). Then \((\beta^3-\beta)\cdot\bsym T=-(\beta^3-\beta)\). It implies \(\beta^3-\beta\in RP_Q\cap\alpha K\).

If \(\alpha^2=\mu\) then we take \(t\) as an uniformaizer of \(Q\). Hence we get
\[
RP_Q=\FF_3\lambda\oplus\bigoplus_{3\nmid j>0}k_Qt^{-j}.
\]
In this case, \(k_Q=k\oplus\alpha k\). For \(r\alpha+s\in k_Q\) (\(r,s\in k\)), we have
\begin{align*}
    \wp(r\alpha+s)&=(r\alpha+s)^3-(r\alpha+s)\\
    &=r^3\alpha^3+s^3-r\alpha-s\\
    &=(r^3\mu-r)\alpha+(s^3-s).
\end{align*}
Hence \(\wp(k_Q)\cap k=\wp(k)\) that is not \(k\). Thus we can choose \(\lambda\) from \(k\). Therefore
\[
RP_Q\cap\alpha K=\bigoplus_{3\nmid j>0}\alpha kt^{-j}.
\]

If \(\alpha\ne\mu\), then we take \(\alpha\) as an uniformizer of \(Q\). Then
\[
RP_Q\cap\alpha K=\bigoplus_{j>0,\gcd(j,6)=1}k\alpha^{-j}.
\]
\begin{lemm}\label{Lem:4.8}
Let \(A\in S_7\) be as above. Let \(e_{L/K}\) be the ramification index of \(L/K\). More precisely, 
\[
e_{L/K}=\begin{cases}
3&(\alpha^2=\mu)\\
6&(\text{otherwise})
\end{cases}.
\]
Then
\[
\bsym v_V(A)=v_K(\alpha^2)+n_1+n_2,
\]
where
\begin{gather*}
    n_1=\ceil{\frac{-v_L(\beta)}{e_{L/K}}},\\
    n_2=\ceil{\frac{-2v_L(\beta)-v_K(\alpha)}{e_{L/K}}}.
\end{gather*}
\end{lemm}
\begin{proof}
Let \(M\) be the linear part of the coordinate ring of \(V\). As \(x_1,x_2,x_3\in M\) is a \(\mathcal O_K\)-basis of \(M\), we take the dual basis \(\varphi_1,\varphi_2,\varphi_3\in\Hom_{\mathcal O_K}(M,\mathcal O_L)\).

We compute the tuning module \(\Xi_{L/K}^{V|_{H_A}}\). For \(c_1,c_2,c_3\in\mathcal O_L\),
\begin{align*}
    &\sum_{i=1}^3c_i\varphi_i\in\Xi_{L/K}^{V|_{H_A}}\\
    \Leftrightarrow&1\leq\forall j\leq 3,\begin{cases}
    (\sum_{i=1}^3c_i\varphi_i)(x_j\cdot\bsym S)=(\sum_{i=1}^3c_i\varphi_i)(x_j)\cdot\bsym S\\
    (\sum_{i=1}^3c_i\varphi_i)(x_j\cdot\bsym T)=(\sum_{i=1}^3c_i\varphi_i)(x_j)\cdot\bsym T
    \end{cases}\\
    \Leftrightarrow&\begin{cases}
    \left(\sum_{i=1}^3c_i\varphi_i\right)(-x_3)=c_1\cdot\tau\\
    \left(\sum_{i=1}^3c_i\varphi_i\right)(-x_2)=c_2\cdot\tau\\
    \left(\sum_{i=1}^3c_i\varphi_i\right)(-x_1)=c_3\cdot\tau
    \end{cases}\\
    \Leftrightarrow&\begin{cases}
    c_1=c_2\cdot\bsym S=c_3\cdot\bsym S^2\\
    c_1=-c_3\cdot\bsym T,\ c_2=-c_2\cdot\bsym T
    \end{cases}\\
    \Leftrightarrow& c_1=c_2\cdot\bsym S,\ c_3=c_2\cdot\bsym S^2,\ c_2\in(K\alpha+K\beta+K\alpha\beta^2)\cap\mathcal O_L.
\end{align*}
As \(v_L(\alpha)=3v_Q(\alpha)\in 3\ZZ\) and \(v_L(\beta)=v_Q(\beta^3-\beta)\not\in3\ZZ\), we have
\[
v_L(\alpha)\not\equiv v_L(\beta)\not\equiv v_L(\alpha\beta^2)\not\equiv v_L(\alpha)\mod 3.
\]
Hence
\[
(K\alpha+K\beta+K\alpha\beta^2)\cap\mathcal O_L=\alpha\mathcal O_K+t^{n_1}\beta\mathcal O_K+t^{n_2}\alpha\beta^2\mathcal O_K.
\]
We put
\begin{gather*}
    \psi_1=\alpha(\varphi_1+\varphi_2+\varphi_3),\\
    \psi_2=t^{n_1}((\beta+1)\varphi_1+\beta\varphi_2+(\beta-1)\varphi_3),\\
    \psi_3=\alpha t^{n_2}((\beta+1)^2\varphi_1+\beta^2\varphi+(\beta-1)^2\varphi_3).
\end{gather*}
Then \((\psi_i)_{i=1}^3\) is an \(\mathcal O_K\)-basis of \(\Xi_{L/K}^{V|_{H_A}}\). Therefore,
\begin{align*}
\bsym v_V(A)&=v_K\left(\alpha^2t^{n_1+n_2}\det\begin{bmatrix}
1&1&1\\
\beta+1&\beta&\beta-1\\
(\beta+1)^2&\beta&(\beta-1)^2
\end{bmatrix}\right)\\
&=v_K(\alpha^2)+n_1+n_2.
\end{align*}
\end{proof}
We set
\[
S_{7,m,j}=\{A\in S_7\mid v_K(\alpha^2)=m,\ v_Q(\beta^3-\beta)=-j\}.
\]
From Lemma \ref{Lem:4.8},
\[
v_V(A)=\begin{cases}
\ceil{\frac j3}+\ceil{\frac{2j}3}&(m=0,j\in\ZZ-3\ZZ)\\
1+\ceil{\frac j6}+\ceil{\frac{2j-3}6}&(m=1,j\in\ZZ,\gcd(j,6)=1)
\end{cases}.
\]
We have
\[
    \sharp S_{7,m,j}=\begin{cases}
    \frac 12(q-1)q^{j-\floor{\frac j3}-1}&(m=0,j\in\ZZ-3\ZZ)\\
    (q-1)q^{j-\floor{\frac j2}-\floor{\frac j3}+\floor{\frac j6}-1}&(m=1,j\in\ZZ,\gcd(j,6)=1)
    \end{cases}.
\]
Therefore
\begin{align}
    \sum_{A\in S_7}\frac{q^{3-\bsym v_V(A)}}{C_G(H_A)}
    &=\sum_{3\nmid j>0}\sum_{A\in S_{7,0,j}}q^{3-\ceil{\frac j3}-\ceil{\frac{2j}3}}\notag\\&\qquad+\sum_{j>0,\gcd(j,6)=1}\sum_{A\in S_{7,1,j}}q^{2-\ceil{\frac j6}-\ceil{\frac{2j-3}6}}\notag\\
    &=\frac 12(q-1)q^2\sum_{3\nmid j>0}q^{j-\floor{\frac j3}-\ceil{\frac j3}-\ceil{\frac{2j}3}}\notag\\&\qquad+(q-1)q\sum_{j>0,\gcd(j,6)=1}q^{j-\floor{\frac j2}-\floor{\frac j3}+\floor{\frac j6}-\ceil{\frac j6}-\ceil{\frac{2j-3}6}}\notag\\
    &=\frac 12(q-1)q^2\sum_{r=0}^\infty(q^{(3r+1)-r-(r+1)-(2r+1)}+q^{(3r+2)-r-(r+1)-(2r+2)})\notag\\&\qquad+(q-1)q\sum_{r=0}^\infty(q^{(6r+1)-3r-2r-1-2r}+q^{(6r+5)-(3r+2)-(2r+1)-1-(2r+2)})\notag\\
    &=\frac 12(q-1)q^2\sum_{r=0}^\infty(q^{-r-1}+q^{-r-1})+(q-1)q\sum_{r=0}^\infty(q^{-r}+q^{-r-1})\notag\\
    &=(q-1)q\sum_{r=0}^\infty q^{-r}+(q-1)q(1+q^{-1})\sum_{r=0}^\infty q^{-r}\notag\\
    &=2q^2+q.\label{eq:4.2.7}
\end{align}
\begin{thm}\label{thm:4.9}
Let \(G\) be a small subgroup isomorphic to \(C_l^2\rtimes\symm\) where \(l\) is a prime. Suppose that \(l\ne 2,3\) and \(q-1\in 2l\ZZ\). For the natural representation \(V\) of \(G\), we get
\[
\sum_{A\in\etset(K)}\frac{q^{3-\bsym v_V(A)}}{\sharp C_G(H_A)}=q^3+\frac{(l+5)(l+7)}{12}q^2+\frac{(l+1)(l+5)}{12}q.
\]
\end{thm}
\begin{proof}
From the formulas (\ref{eq:4.2.1}), (\ref{eq:4.2.2}), (\ref{eq:4.2.3}), (\ref{eq:4.2.4}), (\ref{eq:4.2.5}), (\ref{eq:4.2.6}), and (\ref{eq:4.2.7}), we get
\begin{align*}
    \sum_{A\in\etset(K)}\frac{q^{3-\bsym v_V(A)}}{\sharp C_G(H_A)}&=\sum_{i=1}^7\sum_{A\in S_i}\frac{q^{3-\bsym v_V(A)}}{\sharp C_G(H_A)}\\
    &=\frac 1{6l^2}q^3+(\frac 13q^3+q^2)\\
    &\qquad+\left(\frac{l^2-1}{6l^2}q^3+\frac{(l-1)(l+4)}{12l}q^2+\frac{(l-1)(l-2)}{12l}q\right)\\
    &\qquad+\left(\frac{(l-1)^2(l+4)}{12l}q^2+\frac{(l-1)^2(l-2)}{12l}q\right)\\
    &\qquad+\left(\frac 1{2l}q^3+\frac 1lq^2\right)\\
    &\qquad+\left(\frac{l-1}{2l}q^3+\frac{(l-1)(3l+4)}{4l}q^2+\frac{3(l-1)}4q\right)\\
    &\qquad+(2q^2+q)\\
    &=q^3+\frac{(l+5)(l+7)}{12}q^2+\frac{(l+1)(l+5)}{12}q.
\end{align*}
\end{proof}

If \(l=2\), the set \(\etset(K)\) is divided into eight part which are \(S_1, S_2,\ldots, S_7\) defining in the case of \(l\ne 2\) and
\[
S_8=\{A\in\etset(K)\mid H_A\text{ generated by }\bsym T \text{ and }\diag(-1,0,-1)\}.
\]
For the sums \(\sum_{A\in S_i}\frac{q^{3-\bsym v_V(A)}}{\sharp C_G(H_A)}\) \((i=1,2,3,4,5,7)\), we can use the formulas (\ref{eq:4.2.1}), (\ref{eq:4.2.2}), (\ref{eq:4.2.3}), (\ref{eq:4.2.4}), (\ref{eq:4.2.5}), and (\ref{eq:4.2.7}).

If \(A\in S_6\), then \(A\) corresponds to a cyclic subgroup of \(K^\times/(K^{\times})^4\) by the Kummer theory. All the cyclic subgroup is generated by \(\mu\), \(\mu^it\) \((0\leq i\leq 3)\), or \(\mu t^2\). From Lemma \ref{lem:4.7}, we have \(\bsym v_V(A)=1\) for any \(A\in S_6\) not corresponding to the subgroup generated by \(\mu\). Hence we have
\begin{align}
    \sum_{A\in S_6}\frac{q^{3-\bsym v_V(A)}}{\sharp C_G(H_A)}&=\frac 14 q^3+\frac{\sharp S_6-1}4q^2\notag\\
    &=\frac 14q^3+\frac 54q^2.\label{eq:4.2.8}
\end{align}

If \(A\in S_8\), then \(A=L^{\oplus 6}\) where \(L\) is a \(C_2^2\)-extension of \(K\). From the Kummer theory, \(L\) is generated by \(\alpha,\beta\) satisfying
\[
\alpha^2=\mu, \beta^2=t
\]
over \(K\). Applying Lemma \ref{lem:tameVFct}, we get \(\bsym v_V(A)=1\) because \(\beta\) is a uniformizer of \(L\), the ramification index of \(L/K\) is two,  and the age of matrices in \(\SL_3(K)\) of order two are one. Giving an action of \(G\) on \(A\) corresponds to choosing generators \(h_1,h_2\) of \(H_A\) such that
\begin{align*}
    \alpha\cdot h_1&=-\alpha,&\alpha\cdot h_2&=\alpha\\
    \beta\cdot h_1&=\beta,&\beta\cdot h_2&=-\beta.
\end{align*}
up to conjugate. Then the choices are
\[
(h_1,h_2)=(\bsym T,\diag(-1,1,-1)),\ (\diag(-1,1,-1),\bsym T),\text{ or }(\bsym T,\diag(-1,1,-1)\bsym T).
\]
Hence
\[
\sharp S_8=3.
\]
As \(C_G(H_A)=H_A\), we get
\begin{align}\label{eq:4.2.9}
    \sum_{A\in S_8}\frac{q^{3-\bsym v_V(A)}}{\sharp C_G(H_A)}&=\frac 34q^2.
\end{align}
\begin{thm}
Let \(G\) be a small subgroup isomorphic to \(C_2^2\rtimes\symm\). Suppose that  \(q-1\in 4\ZZ\). For the natural representation \(V\) of \(G\), we get
\[
\sum_{A\in\etset(K)}\frac{q^{3-\bsym v_V(A)}}{\sharp C_G(H_A)}=q^3+6q^2+q.
\]
\end{thm}
\begin{proof}
By using the formulas the formulas (\ref{eq:4.2.1}), (\ref{eq:4.2.2}), (\ref{eq:4.2.3}), (\ref{eq:4.2.4}), (\ref{eq:4.2.5}), (\ref{eq:4.2.7}), (\ref{eq:4.2.8}), and (\ref{eq:4.2.9}), we get
\begin{align*}
    \sum_{A\in\etset(K)}\frac{q^{3-\bsym v_V(A)}}{\sharp C_G(H_A)}
    &=\frac 1{24}q^3+\left(\frac 13q^3+q^2\right)+\left(\frac 3{24}q^3+\frac 14q^2\right)+\frac 14q^2\\
    &\qquad+\left(\frac 14q^3+\frac 12q^2\right)+(2q^2+q)+\left(\frac 14q^3+\frac 54q^2\right)+\frac 34q^2\\
    &=q^3+6q^2+q.
\end{align*}
\end{proof}

%--------------------------------------------------------------------------

\section{Euler characteristic}

In this section, we compute the Euler characteristic, which defined from the \(l\)-adic cohomology, of a crepant resolution \(Y\) of \(V/G\) for a representation \(V\) of a cyclic or symmetric type subgroup \(G\subset\SL_3(k)\). 

To compute the Euler characteristic, we use the following well-known formula.
\begin{prop}
For a smooth variety \(Y\) over \(\FF_q\). Suppose that
\[
\sharp Y(\FF_{q^m})=\sum_{i=1}^na_iq^{im}
\]
where \(a_i\in\ZZ\). Then
\[
\chi(Y)=\sum_{i=1}^na_i.
\]
\end{prop}
\begin{proof}
Let \(Z(t)\in\QQ[[t]]\) be the zeta function of \(Y\), which is defined as
\[
Z(t)=\exp\left(\sum_{m=1}^\infty\sharp Y(\FF_{q^m})\frac{t^m}m\right),
\]
or equivalently
\[
\frac d{dt}\log Z(t)=\sum_{m=1}^\infty\sharp Y(\FF_{q^m})t^{m-1}.
\]
From the assumption, we have
\begin{align*}
\sum_{m=1}^\infty\sharp Y(\FF_{q^m})t^{m-1}&=\sum_{m=1}^\infty\sum_{i=1}^na_iq^{im}t^{m-1}\\
&=\sum_{i=1}^na_iq^i\sum_{m=1}^\infty (q^it)^{m-1}\\
&=\sum_{i=1}^n\frac{a_iq^i}{1-q^it}.
\end{align*}
Thus
\[
\log Z(t)=\int\sum_{i=1}^n\frac{a_iq^i}{1-q^it}dt=\sum_{i=1}^n-a_i\log(1-q^it),
\]
and hence
\[
Z(t)=\exp(\sum_{i=1}^n-a_i\log(1-q^it))=\prod_{i=1}^n\frac 1{(1-q^it)^{a_i}}.
\]
From the Weil conjecture we get
\[
\chi(Y)=\deg\left(\prod_{i=1}^n(1-q^it)^{a_i}\right)=\sum_{i=1}^na_i.
\]
\end{proof}
\begin{coro}
Let \(G\) be a cyclic type or symmetric type subgroup of \(\SL_3(k)\). For a crepant resolution \(Y\to V/G\) of a quotient of a canonical representation \(V\) of \(G\), we have
\[
\chi(Y)=\begin{cases}
3+\frac{l-1}3&(G\cong C_l\rtimes\ZZ/3\ZZ)\\
3+\frac{l^2-1}3&(G\cong(C_l)^2\rtimes\ZZ/3\ZZ)\\
\frac{(l-1)(l-2)}6+2l+4&(G:\text{symmetric type})
\end{cases}
\]
\end{coro}
\begin{proof}
By the property of the stringy-point count, we have
\[
\sharp Y(\FF_q)=\sharp_{st}V/G.
\]
Hence the assertion follows from the previous proposition and Theorem \ref{MT}.
\end{proof}
\begin{rem}
From \cite[Theorem 1.2]{Yam}, the quotient varieties \(V/G\) have crepant resolutions. 
\end{rem}
%---------------------------------------------------------------------------

\end{document}